\documentclass[11pt,a4paper]{amsart}

\usepackage{amsmath,amsthm,amssymb,latexsym,a4wide,tikz,multicol,tikz-cd}
\usepackage{mathtools,stmaryrd}
\usepackage{enumerate}

\newtheorem{theorem}{Theorem} [section]
\newtheorem{lemma}[theorem]{Lemma}

\newtheorem{proposition}[theorem]{Proposition}

\newtheorem{example}[theorem]{Example}
\newtheorem{remark}[theorem]{Remark}
\newtheorem{definition}[theorem]{Definition}

\newcommand{\dom}{\mathrm{dom}}
\newcommand{\ran}{\mathrm{ran}}
\newcommand{\id}{\mathrm{id}}


\usepackage[active]{srcltx}

\usepackage{enumerate}
\numberwithin{equation}{section}

\title[Boolean inverse semigroups and their type monoids]{Boolean inverse semigroups and their type monoids}

\author{Ganna Kudryavtseva}
\address{G. Kudryavtseva: University of Ljubljana,
Faculty of Mathematics and Physics, Jadranska ulica~19, SI-1000 Ljubljana, Slovenia / Institute of Mathematics, Physics and Mechanics, Jadranska ulica 19, SI-1000 Ljubljana, Slovenia}
\email{ganna.kudryavtseva\symbol{64}fmf.uni-lj.si}

\thanks{The author was supported by the ARIS grant P1-0288.}

\begin{document}

\maketitle

\begin{abstract}  
This is an expository paper which provides a quick introduction to Boolean inverse semigroups and their type monoids, with the emphasis on techniques and insights of the theory, and also treats the connection of the type monoid ${\mathrm{Typ}}(S)$ of a Boolean inverse semigroup $S$ with the monoid $V(K\langle S\rangle)$ of the ring $K\langle S\rangle$ assigned to $S$. We give original direct and simple proofs of some known results, such as the structure of semisimple  Boolean inverse semigroups and the presentation of the type monoid by generalized rook matrices. We also prove that the type monoid of the tight Booleanization of a graph inverse semigroup is isomorphic to the graph monoid of this graph. 
\end{abstract}

\section{Introduction}
Boolean inverse semigroups are non-commutative generalizations of generalized Boolean algebras and are inverse semigroups whose idempotent semilattice has a structure of a generalized Boolean algebra, and also satisfy an additional condition that compatible joins exist (see Subsection \ref{subs:bis}). It is remarkable that the classical Stone duality between generalized Boolean algebras and locally compact Stone spaces admits a natural extension to the non-commutative setting of Boolean inverse semigroups which form a category dually equivalent to the category of Stone groupoids, known also as ample groupoids. This result is due to Lawson \cite{L10}, see also \cite{KL17, L12, LL13}, and develops the ideas of an earlier work of Resende \cite{R07} which established a correspondence between pseudogroups (which are analogues of Boolean inverse semigroups, whose idempotents form a frame) and suitable localic \'etale groupoids. Later on, Resende's results were extended to a categorical duality (in fact, varying morphisms, to several dualities) in \cite{KL17,LL13}.  Moreover, \cite{KL17} also establishes an adjunction between localic \'etale groupoids and topological \'etale groupoids extending the classical adjunction \cite{JS} between locales and topological spaces (and also a generalization of these results from inverse to two-sided restriction semigroups and from groupoids to categories). 

The non-commutative Stone duality makes Boolean inverse semigroups algebraic counterparts of ample groupoids and a part of the broad research stream, which relates inverse semigroups, \'etale groupoids, $C^*$-algebras \cite{Exel:comb,Pat} and Steinberg algebras (originated in \cite{S10}, see also the excellent survey paper \cite{CH}). Boolean inverse semigroups can be looked at as abstract analogues of inverse semigroups of compact and open local bisections of an ample groupoid. Important properties of the groupoid, such as being Hausdorff, effective, minimal or topologically free are equivalent to suitable algebraic properties of its Boolean inverse semigroup (see \cite[Proposition 3.7]{S10} and \cite[Proposition 2.10]{StSz}). Steinberg showed in \cite{S10} that, for an inverse semigroup $S$, the semigroup algebra $KS$, over a commutative ring $K$ with a unit, is isomorphic to the groupoid algebra $K{\mathcal G}S$ of the universal groupoid  ${\mathcal G}S$ (see \cite{Pat}) of $S$, nowadays called the {\em Steinberg algebra} of ${\mathcal G}S$. The Boolean inverse semigroup ${\mathcal B}(S)$ of compact and open local bisections of ${\mathcal G}S$ is called the {\em Booleanization} of $S$, and one can show that the algebra $K{\mathcal G}S$ is isomorphic to the algebra $K_0{\mathcal B}_{tight}(S)$ of a certain quotient ${\mathcal B}_{tight}(S)$ of  ${\mathcal B}(S)$, called the {\em tight Booleanization} of $S$ (see \cite{StSz} for details and \cite{K20} for a more general notion). Boolean inverse semigroups seem to be sometimes a convenient replacement for their underlying ample groupoids  to work with, which we illustrate by the proof of Theorem \ref{th:isom_graph}, that should be compared with the proof of \cite[Theorem 7.5]{ABPS21}.  We also mention the paper \cite{L21a} by Lawson which reinterprets, via Boolean inverse semigroups, the Matui's spatial realization theorem~\cite{M15}.

The first aim of this paper is to give a quick introduction into techniques and insights, needed to start working with Boolean inverse semigroups. That is why we included some of the basic (but important) results with proofs (which are mostly original and are often shorter than those to be found in the literature).  Section~\ref{s:inv} introduces inverse semigroups, with the emphasis on the notions (such as compatibility, orthogonality, fundamentality) needed in the sequel. Section~\ref{s:boolean} brings  Boolean inverse semigroups on the stage and treats some basic notions and results of their theory, such as morphisms, ideals and congruences.

The scope of this paper did not allow us to give a detailed treatment of some important aspects of the theory, such as the duality between Boolean inverse semigroups ample groupoids (see Remark \ref{rem:duality}).  We did, however, develop the duality (at the level of objects) for semisimple Boolean inverse semigroups from scratch  and using elementary tools (see Section~\ref{s:semisimple}), which should hopefully convey the essence of the general case, it roughly only remains to make (an upgrade of) a standard passage from a finite Stone duality to a general one via replacing atoms by ultrafilters and bringing in a suitable topology. The presented duality can be also derived as a special case of the Stone duality for Boolean inverse semigroups, see \cite{L21} for details. 

The other focus of the paper is on the type monoid of a Boolean inverse semigroup, introduced in \cite{Wal} (see also \cite{KLLR}) which is an analogue of the monoid $V(R)$ associated to a ring $R$ (see, e.g., \cite{ABP20, W17} and references therein). Various aspects of the theory of type monoids of Boolean inverse semigroups (and also of Boolean inverse semigroups {\em per se}) are studied in depth in the book by Wehrung \cite{W17}. 
Section \ref{s:type} opens with an introduction to partial commutative monoids, refinement monoids and proceeds to the definition of the partial commutative monoid ${\mathrm{Int}}(S)$, associated to a Boolean inverse semigroup $S$, and its enveloping monoid ${\mathrm{Typ}}(S)$, the {\em type monoid} of $S$. In Theorem~\ref{th:type} we provide a direct proof (of a result by Wehrung) that ${\mathrm{Typ}}(S)$ is isomorphic to the type monoid ${\mathrm{Typ}(M_{\omega}(S)}) = {\mathrm{Int}}(M_{\omega}(S))$ of  the Boolean inverse semigroup $M_{\omega}(S)$ of generalized rook matrices over $S$, see Subsection~\ref{subs:rook} for details. In Section \ref{s:connection} we present the canonical map from ${\mathrm{Typ}}(S)$ to $V(K\langle S\rangle)$, the monoid associated with the Steinberg algebra $K\langle S\rangle$, and prove (see Proposition \ref{prop:lm}) that if $S$ is locally matricial, this map is an isomorphism (cf. \cite[Proposition 6.7.3]{W17}). Finally, in Section~\ref{s:graph} we use Boolean inverse semigroups to prove that ${\mathrm{Typ}}({\mathcal B}_{tight}(I(\Gamma)))$ where $I(\Gamma)$ is a graph inverse semigroup assigned to a  row-finite graph $\Gamma$, is canonically isomorphic to the graph monoid $M_{\Gamma}$ (see Theorem \ref{th:isom_graph}).

Boolean inverse semigroups have great potential to find further applications in the theory of \'etale groupoids and related areas, for example, topological full groups, Steinberg algebras and groupoid $C^*$-algebras, which are at a crossway of several branches of algebra and analysis and attract attention of authors with various backgrounds. This paper makes an accent on techniques and ideas supplied by inverse semigroup theory and 
is intended to be useful for beginners, and also for experts with with background aside from semigroup theory. 

\section{Inverse semigroups}\label{s:inv}

\subsection{Definition and first properties.} An element $e$ of a semigroup $S$ is called an {\em idempotent} if $e^2=e$, and  the set of all idempotents of $S$ is denoted by $E(S)$. For $a,b\in S$ we say that $b$ is an {\em inverse} of $a$  if $aba=a$ and $bab=b$. (Note that $a$ has an inverse if there is $b\in S$ such that $a=aba$, then $bab$ is an inverse of $a$). The semigroup $S$ is called {\em regular} if every $a\in S$ has at least one inverse. We say that $S$ is an {\em inverse semigroup} if every its element $a$ has precisely one inverse, denoted $a^{-1}$. It is well known that $S$ is an inverse semigroup if and only if it is regular and its idempotents commute. 

The prototypical example of an inverse semigoroup is the {\em symmetric inverse semigroup} ${\mathcal I}(X)$ on a set $X$. Its elements are all injective maps $Z\to X$ where $Z\subseteq X$. Such maps are called {\em partial permutations} of $X$. The binary multiplicaiton operation in ${\mathcal I}(X)$ is the natural {\em composition} of partial maps. For a partial permutation $f\colon Z\to X$, where $Z\subseteq X$, we say that $Z$ is the {\em domain} of $f$ and denote it by $\dom f$. The set $f(Z)$ is called the {\em range} (or the {\em image}) of $f$ and is denoted by $\ran f$. The map $f$ is thus a bijection $f\colon \dom f\to \ran f$ and its inverse $f^{-1}\colon \ran f\to \dom f$ is the inverse bijection to $f$. The idempotents of ${\mathcal I}(X)$ are precisely the identity maps on subsets of $X$ and $\id_Y\id_Z = id_{Y\cap Z} = \id_Z\id_Y$. Every inverse semigroup $S$ can be embedded into the symmetric inverse semigroup ${\mathcal I}(S)$. This fundamental results, known as the {\em Wagner-Preston\footnote{Mentioning Wagner first was suggested to me by Mark Lawson and acknowledges the fact that Wagner discovered the result before Preston.}}  theorem, shows that inverse semigroups can be looked at as inverse semigroups of {\em partial permutations}. 

If $S$ is an inverse semigroup and $a\in S$, the elements $a^{-1}a$ and $aa^{-1}$ are idempotents called the {\em domain idempotent} and the {\em range idempotent} of $a$ and are often denoted by ${\bf d}(a)$ and ${\bf r}(a)$, respectively. The {\em natural partial order} on $S$ is defined by $a\leq b$ if there is $e \in E(S)$ satisfying $a=be$. It is easy to show that $a\leq b$ holds if and only if $a=eb$ for some $e\in E(S)$ if and only if $a=b{\bf d}(a)$ if and only if $a = {\bf r}(a)b$. If we think of $S$ as of an inverse subsemigroup of some ${\mathcal I}(X)$, then $a\leq b$ just means that the partial permutation $a$ is a restriction of $b$.

Observe that for $a\in S$ and $e\in E(S)$ we have $a^{-1}ea = (ea)^{-1}(ea) = {\bf d}(ea)$, so that $a^{-1}ea$ is an idempotent.
It is easy to see that in an arbitrary inverse semigroup we have:
$$
{\bf d}(ab) = {\bf d}({\bf d}(a)b), \,\, {\bf r}(ab) = {\bf r}(a{\bf r}(b)).
$$
From the compatibility of the natural partial order with the multiplication and the inversion operations it follows that $a\leq b$ implies ${\bf d}(a)\leq {\bf d}(b)$ and similarly for ranges.

Any group is an inverse semigroup, where the inverse is the usual group inverse, and groups are precisely the inverse semigroups with one idempotent. 

Recall that the {\em least upper bound} or the  {\em join} of elements $a$ and $b$ of a poset is the element $c$ such that $a,b\leq c$, that is, $c$ is an {\em upper bound} of $a$ and $b$ and, in addition, $a,b\leq d$ implies $c\leq d$. The {\em greatest upper bound} or the {\em meet} is defined dually.
A {\em semilattice} is an inverse semigroup satisfying $S=E(S)$. If $E$ is a semilattice, for any $e,f\in E$ their product $ef$ is the  meet $e\wedge f$ with respect to the natural partial order. This leads to the equivalent characterization of semilattices as posets with binary meets. Groups and semilattices are thus two extreme classes of inverse semigroups (and a wide and important class of inverse semigroup, $E$-unitary inverse semigroups, can be constructed from groups and semilattices using the construction of the partial action product \cite{KelL}). In the sequel we will need the following observation.

\begin{lemma}\label{lem:dom1} ${\bf d}(ab)\leq {\bf d}(b)$ and ${\bf r}(ab) \leq {\bf r}(a)$.
\end{lemma}
\begin{proof}
We have ${\bf d}(ab) {\bf d}(b) = (ab)^{-1}ab b^{-1}b = (ab)^{-1}ab = {\bf d}(ab)$. The second inequality follows from the first one applying ${\bf r}(a) = {\bf d}(a^{-1})$.
\end{proof}

We say that elements $a$ and $b$ are {\em compatible}, denoted $a\sim b$, if $a^{-1}b$ and $ab^{-1}$ are idempotents, and that they are {\em orthogonal}, denoted $a\perp b$, if $a^{-1}b = ab^{-1} = 0$. Note that $a\perp b$ implies $a\sim b$ and any two idempotents are compatible. Two idempotents $e$ and $f$ are orthogonal if and only if $ef=0$. If $s\perp t$, we will  write $s\oplus t$ instead of $s\vee t$ to emphasise that the join is orthogonal. Two partial permutations $a,b\in {\mathcal I}(X)$ are compatible if and only if their union $a\cup b$ is a partial permutation, which then coincides with the join $a\vee b$. The next well known lemma provides a handy characterization of compatibility avoiding referring to inverse elements.

\begin{lemma}\label{lem:comp} \mbox{}
\begin{enumerate}
\item $a^{-1}b$ is an idempotent if and only if ${\bf r}(a)b = {\bf r}(b)a$.
\item $ab^{-1}$ is an idempotent if and only if $b{\bf d}(a) = a{\bf d}(b)$.
\item $a\sim b$ if and only if ${\bf r}(a)b = {\bf r}(b)a$ and $b{\bf d}(a) = a{\bf d}(b)$.
\end{enumerate}
\end{lemma}

\begin{proof} (1) Suppose that $e=a^{-1}b \in E(S)$. Then $ae\leq a$, that is, ${\bf r}(a)b\leq a$. If we multiply this inequality form the left with ${\bf r}(b)$, we obtain
${\bf r}(a) b\leq {\bf r}(b)a$. The opposite inequality follows by symmetry. Conversely, suppose that ${\bf r}(a)b = {\bf r}(b)a$,
that is, $aa^{-1}b = bb^{-1}a$. Multiplying this from the left with $a^{-1}$, we obtain $a^{-1}b = a^{-1}bb^{-1}a = a^{-1}{\bf r}(b) a$ which is an idempotent as an element conjugate to an idempotent.  Part (2) follows from by symmetry and 
part (3) from the first two parts.
\end{proof}

Note that, if $a$ and $b$ have an upper bound, $c$, then then $a{\mathbf d}(b) = c {\mathbf d}(a) {\mathbf d}(b) = b {\mathbf d}(a)$ and similarly ${\bf r}(a)b = {\bf r}(b)a$, so $a\sim b$. 
Consequently, a pair of non-compatible elements can not have a join. From $a\sim b$, it, however does not follow that $a\vee b$ exists. For example, in the singular part ${\mathcal I}_n\setminus {\mathcal S}_n$ (where ${\mathcal S}_n$ is the symmetric group) of ${\mathcal I}_n$ let $a,b$ be compatible such that their join in ${\mathcal I}_n$ is a permutation. Then $a$ and $b$ do not have an upper bound in ${\mathcal I}_n\setminus {\mathcal S}_n$. For an example of two compatible elements in an inverse semigroups, which have an upper bound, but do not have the join, see Example \ref{ex:a}.

\begin{lemma}\label{lem:ort}\mbox{}
\begin{enumerate}
\item $a^{-1}b = 0$ if and only if ${\bf r}(a) \perp {\bf r}(b)$.
\item $ab^{-1} = 0$ if and only if ${\bf d}(a) \perp {\bf d}(b)$.
\item $a\perp b$ if and only if ${\bf r}(a) \perp {\bf r}(b)$ and ${\bf d}(a) \perp {\bf d}(b)$.
\end{enumerate}
\end{lemma}

\begin{proof} If $a^{-1}b = 0$ then ${\bf r}(a){\bf r}(b) = a(a^{-1}b)b^{-1} = 0$. If ${\bf r}(a)\perp {\bf r}(b)$ then multiplying $aa^{-1}bb^{-1}=0$ from the left with $b^{-1}$ and from the right with $b$ we get $b^{-1}b=0$. Part (1) follows. Part (2) follows by symmetry. Part (3) follows from parts (1) and (2).
\end{proof}

Let $S$, $T$ be inverse semigroups and $\varphi\colon S\to T$ a semigroup homomorphism. Note that for any $s\in S$ the elements $\varphi(s)$ and $\varphi(s^{-1})$ are mutually inverse. By the uniqueness of the inverses it follows that $\varphi(s^{-1}) = \varphi(s)^{-1}$, that is, a semigroup homomorphism between inverse semigroups is automatically an {\em inverse semigroup homomorphism.} If $S$ and $T$ have zero elements, we will restrict our attention to homomorphisms which preserve the zero. 

\subsection{Fundamental inverse semigroups.} Here we present a brief account on fundamental inverse semigroups, for a slightly different exposition see \cite{Lawson_book}. A congruence $\tau$ on an inverse semigroup $S$ is called {\em idempotent-separating} if $e \mathrel{\tau} f$ with $e,f\in E(S)$ implies $e=f$. Define the relation $\mu$ on $S$ by $a\mathrel{\mu} b$ if $a^{-1}ea = b^{-1}eb$ or, equivalently, ${\bf d}(ea) = {\bf d}(eb)$ for all $e\in E(S)$.

\begin{lemma} $\mu$ is the maximal idempotent-separating congruence on $S$.
\end{lemma}

\begin{proof} It is easy to see that $\mu$ is a congruence: if $a\mathrel{\mu} b$ then ${\bf d}(eac) = {\bf d}({\bf d}(ea)c) = {\bf d}({\bf d}(eb)c) = {\bf d}(ebc)$ for all $c\in S$ and $e\in E(S)$, so that $ac \mathrel{\mu} bc$. Also  ${\bf d}(cea) = {\bf d}({\bf d}(c)ea) = {\bf d}({\bf d}(c)eb) = {\bf d}(ceb)$ for all $c\in S$ and $e\in E(S)$, so that $ca \mathrel{\mu} cb$. 
Suppose $e,f\in E(S)$ and $e\mathrel{\mu} f$. Since $e^{-1}fe = f^{-1}ff$ rewrites to $ef = f$, we see that $f\leq e$. By symmetry, also $e\leq f$, so that $e=f$. Hence $\mu$ is  idempotent-separating.

Let $\rho$ be an idempotent-separating congruence, $a\mathrel{\rho} b$ and $e\in E(S)$. Then $ea \mathrel{\rho} eb$ so that ${\bf d}(ea) \mathrel{\rho} {\bf d}(eb)$ which implies ${\bf d}(ea) = {\bf d}(eb)$ thus $a\mathrel{\mu} b$. Hence $\rho\subseteq \mu$, and $\mu$ is maximal, as needed.
\end{proof}

\begin{definition} {\em (Fundamental inverse semigroups) An inverse semigroup $S$ is called {\em fundamental} provided that $\mu$ is the equality relation.}
\end{definition}

The definition implies that for any inverse semigroup $S$ the quotient $S/\mu$ is fundamental.

The {\em centralizer} $Z(A)$ of a subset $A$ of $S$ is the set of those $s\in S$ which commute with all elements of $A$. Clearly $Z(E(S)) \supseteq E(S)$ for any inverse semigroup $S$.

\begin{proposition} \label{prop:fund} The following statements are equivalent for an inverse semigroup $S$:
\begin{enumerate}
\item $S$ is fundamental.
\item $Z(E(S)) = E(S)$.
\end{enumerate}
\end{proposition}

\begin{proof} Suppose that $S$ is fundamental and let $a\in Z(E(S))$. Then for each $e\in E(S)$ we have $a^{-1}ea = {\bf d}(a)e$ which yields $a\mathrel{\mu} {\bf d}(a)$. It follows $a={\bf d}(a)\in E(S)$.

Conversely, suppose that $Z(E(S)) = E(S)$. Let first $a \mathrel{\mu} e$ with $e\in E(S)$. Then ${\bf d}(a) \mathrel{\mu} e$, so that ${\bf d}(a) = e$. Since $a\mathrel{\mu} {\bf d}(a)$, for any $f\in E(S)$ we have $a^{-1}fa = {\bf d}(a) f$, so, multiplying by $a$ from the left, we get ${\bf r}(a)fa = a{\bf d}(a) f$, which rewrites to $fa = af$. The assumption yields $a\in E(S)$, so that $a={\bf d}(a)$. Let now $a\mathrel{\mu} b$. Then $a^{-1}a \mathrel{\mu} a^{-1}b$, which implies $a = aa^{-1}a = aa^{-1} b \leq b$. By symmetry, we also have $b\leq a$, so that $a=b$, and $\mu$ is the identity relation.
\end{proof}

\section{Boolean inverse semigroups} \label{s:boolean}
\subsection{Generalized Boolean algebras} 
We will always consider lattices possessing a bottom element $0$ but not necessarily a top element.
A lattice $L$ is called {\em distributive} if it satisfies the distributivity identity $x\wedge (y\vee z) = (x\wedge y)\vee (x\wedge z)$ or the equivalent identity $x\vee (y\wedge z) = (x\vee y)\wedge (x\vee z)$.
A {\em generalized Boolean algebra} is a relatively complemented distributive lattice. This means that for any elements $a,b$ such that $a\leq b$ there is a (necessarily unique) {\em relative complement} of $a$ with respect to $b$, that is, an element $b\setminus a$, satisfying $a \vee (b\setminus a) = b$ and $a\wedge (b\setminus a) = 0$. 

\subsection{Joins in inverse semigroups}
When discussing joins  (and meets) in inverse semigroups, we always consider them with respect to the natural partial order.
Let $S$ be an inverse semigroup and suppose that $a,b\in S$ are such that $a\vee b$ exists in $S$. Since $a,b\leq a\vee b$ it follows that $a$ and $b$ are necessarily compatible.

\begin{proposition}\label{prop:join} \cite[Proposition 17, Section 1.4]{Lawson_book} Suppose that $a\sim b$ and $a\vee b$ exists in $S$. Then  ${\bf d}(a)\vee {\bf d}(b)$ exists in $S$ and ${\bf d}(a\vee b) = {\bf d}(a)\vee {\bf d}(b)$.
\end{proposition}

\begin{proof} Let $c = a\vee b$. Since $a,b\leq c$, we have ${\bf d}(a), {\bf d}(b) \leq {\bf d}(c)$, so ${\bf d}(c)$  is an upper bound of ${\bf d}(a)$ and  ${\bf d}(b)$. To show that it is the least upper bound we suppose that ${\bf d}(a), {\bf d}(b) \leq \gamma$ and aim to show that ${\bf d}(c)\leq \gamma$. Observe that $a = a{\bf d}(a) \leq a\gamma \leq c\gamma$ and similarly $b\leq c\gamma$. We have $c\leq c\gamma$, so that $c={\bf r}(c) c\gamma = c\gamma$, thus ${\bf d}(c) = c^{-1}c = c^{-1}c\gamma = {\bf d}(c)\gamma \leq \gamma$.
\end{proof}

Note that the join ${\bf d}(a) \vee {\bf d}(b) = {\bf d}(a) \vee_S {\bf d}(b)$ in the above proposition is taken in $S$ and one may wonder if the join ${\bf d}(a) \vee_{E(S)}{\bf d}(b)$ exists. This is answered in the next lemma.

\begin{lemma} \label{lem:join} Let $e,f\in E(S)$.  If  $e \vee_S f$ exists, so does $e\vee_{E(S)} f$ and $e\vee_{E(S)} f = e \vee_S f$. 
\end{lemma}

\begin{proof} Supose that $e\vee_S f$ exists and equals $s$. We show that $s$ is an idempotent. By Proposition \ref{prop:join} we have ${\bf d}(e) \vee_S {\bf d}(f) = {\bf d}(e\vee_S f)$. But $e$ and $f$ are idempotents, so $e = {\bf d}(e)$ and $f={\bf d}(f)$. Hence $e\vee_S f= {\bf d}(e\vee_S f)$. Since the element on the right-hand side is an idempotent, so is the element on the left-hand side, that is, $e\vee_S f\in E(S)$. We have that $e\vee_S f$ is the upper bound of $e$ and $f$ in $E(S)$. If $p\in E(S)$ is another their upper bound, then $p$ is also their upper bound in $S$, thus $e\vee_S f\leq p$, which finishes the proof.
\end{proof}

We remark, however, that for $e,f\in E(S)$ the existence of $e\vee_{E(S)} f$ does not imply the existence of $e\vee_S f$. This is shown by the following example.

\begin{example}\label{ex:a}
{\em Let $E = \{0,a,b,1\}$ be a $4$-element Boolean algebra with atoms $a$ and $b$, and let $u\not\in E$. On $S=E\cup \{u\}$ we define the multiplication by $u^2=1$ and so that any product involving $u$ and any of the elements $0,a,b$ is equal to the same product where $u$ is replaced by $1$,  for example, $au=a$. Then $S$ is an inverse semigroup with $E(S)=E$ and is a union of groups. We claim that the join $a\vee_S b$ does not exist, despite the fact that $a\vee_{E(S)} b$ exists (and equals $1$). We have  $1\geq a,b$ and also $u\geq a,b$ because $a= ua = u{\bf d}(a)$ and $b=ub =  u{\bf d}(b)$. So $a$ and $b$ have two common upper bounds in $S$. Since $1\not \geq u$ and $u\not\geq 1$, neither of them is a join of $a$ and $b$. So $a\vee_S b$ does not exist.}
\end{example}

\subsection{Boolean inverse semigroups and their first properties}\label{subs:bis}
The restriction of the natural partial order $\leq$ to $E(S)$ will be denoted by the same symbol $\leq$.
If an inverse semigroup has a zero element, $0$, this element is the minimum one with respect to the natural partial order. Indeed, $0=s\cdot 0 \leq s$ for all $s\in S$.
In particular, $0$ is the minimum idempotent. (However, if $S$ has a minimum idempotent, it need not have a zero element. This happens if $S$ is a group of order at least $2$.)
\begin{definition} \label{def:bis} (Boolean inverse semigroups) \label{def:bool} {\em A {\em Boolean inverse semigroup} is an inverse semigroup with zero which satisfies the following two axioms:
 \begin{enumerate}
 \item[(BIS1)] $(E(S),\leq)$ is a generalized Boolean algebra;
 \item[(BIS2)] For any compatible elements $a,b\in S$ there is their join $a\vee b$ in $S$. 
 \end{enumerate}
  }
 \end{definition}
 
Note that the existence of the zero in $S$ does not follow from (BIS1) and (BIS2). Indeed, any group $G$ satisfies these two axioms, in particular (BIS2) holds vacuously as no pair of distinct elements is compatible. But if $|G|\geq 2$, $G$ does not have a zero. In what follows we always assume that all inverse semigroups have a zero, and we denote it by $0$.

Applying induction, it is easy to show that (BIS2) implies that joins exist for any finite sets of pairwise compatible elements. Furthermore, in a Boolean inverse semigroup the multiplication distributes over binary (and thus also any finite) non-empty compatible joins in that $a(b\vee c) = ab\vee ac$ and $(b\vee c)a = ba\vee ca$ for any $a,b,c\in S$ with $b\sim c$ (see \cite[Proposition 1.4.20]{Lawson_book} and \cite[Proposition 3.1.9]{W17}). Moreover, $b\perp c$ yields $ab\perp ac$ and $ba \perp ca$, so that we have $a(b\oplus c) = ab \oplus ac$ and $(b\oplus c)a = ba\oplus ca$. The following axiom, which is clearly a consequence of (BIS2), is in fact equivalent to (BIS2).

\begin{enumerate}
\item[(BIS2a)] For any orthogonal elements $a,b\in S$ there is their join $a\oplus b$ in $S$.
\end{enumerate}

\begin{lemma}
In Definition \ref{def:bool} condition (BIS2) can be replaced by (BIS2a).
\end{lemma}

\begin{proof}
We suppose that (BIS1) and (BIS2a) hold and show that (BIS2) holds. Let $a\sim b$. Then $a{\bf d}(b) = b{\bf d}(a) = a\wedge b$, so that $a = (a\wedge b) \oplus c$ where $c=a({\bf d}(a)\setminus {\bf d}(b))$. Similarly $b= (a\wedge b) \oplus d$ where $d=b({\bf d}(b)\setminus {\bf d}(a))$. Note that $c\perp d$. By (BIS2a) the join $s=(a\wedge b) \oplus c \oplus d$ exists. We have $s\geq a,b$ and if $t\geq a,b$ then $t\geq (a\wedge b), c,d$ so that $t\geq s$ which proves that $s=a\vee b$.
\end{proof} 

\begin{definition} {\em (Additive homomorphisms) A semigroup homomorphism $\varphi\colon S\to T$ between Boolean inverse semigroups is called {\em additive} provided that $\varphi(s\oplus t) = \varphi(s)\oplus \varphi(t)$ for any orthogonal elements $s,t\in S$.}
\end{definition}

Let $S$ and $T$ be Boolean inverse semigroups. Note that if $s$ and $t$ are compatible then so are $\varphi(s)$ and $\varphi(t)$ for any semigroup homomorphism $\varphi\colon S\to T$. In addition, using $\varphi(0)=0$, it follows that the orthogonality is also preserved.
It is easy to see that a homomorphism $\varphi\colon S\to T$ is additive if and only $\varphi(s\vee t) = \varphi(s)\vee \varphi(t)$ for any compatible $s$ and $t$.  
For $u\leq s$ we put $s\setminus u=s({\bf d}(s)\setminus {\bf d}(u))$ and note that $s\setminus u$ is the only element satisfying $s = (s\setminus u) \oplus (s\wedge u)$. Put differently, the domain map $u\mapsto {\bf d}(u)$ is the Boolean algebra isomorphism between the Boolean algebras $s^{\downarrow}$ and ${\bf d}(s)^{\downarrow}$ whose inverse is given by $e\mapsto se$, and $s\setminus u$ is the inverse image under this isomorphism of ${\bf d}(s) \setminus {\bf d}(u)$. 

Wehrung showed \cite{W17} that Boolean inverse semigroups, considered as algebraic structures in the extended signature $(\cdot, ^{-1}, \obslash,\triangledown)$, where the binary operations $\obslash$ and $\triangledown$ on a Boolean inverse semigroup are defined by $a\obslash b = ({\bf r}(a)\setminus {\bf r}(b)) a ({\bf d}(a)\setminus {\bf d}(b))$ and $a\triangledown b= (a\obslash b) \oplus b$, called the {\em skew difference} and the {\em left skew join}, respectively, form a variety of algebras (these algebras are called biases in \cite{W17}) and that homomorphisms of Boolean inverse semigroups considered 
in the extended signature are precisely additive homomophrisms. By an {\em additive congruence} on $S$ we mean the kernel of an additive homomorphism, that is, it is an equivalence relation on $S$ which a semigroup congruence and also respects the operations, $\obslash$ and $\triangledown$. It follows from \cite[Proposition 3.4.1]{W17} that a semigroup congruence $\sigma$ is additive if and only if for all orthogonal $e,f\in E(S)$ the join of $[e]$ and $[f]$ exists in $S/\sigma$ and equals $[e\oplus f]$.

\subsection{Fundamental Boolean inverse semigroups} 
The following is \cite[Proposition 3.4.5]{W17} but we provide an alternative short proof, for the reader's convenience.

\begin{proposition}\label{prop:mu} Let $S$ be a Boolean inverse semigroup. Then $\mu$ is an additive congruence. That is, $S/\mu$ is a Boolean inverse semigroup and the canonical projection $\mu^{\natural}$ is an additive homomorphism.
\end{proposition}

\begin{proof} We denote the elements of $S/\mu$ by $[a] = \{a' \in S\colon a\mathrel{\mu} a'\}$. Let $[a], [b] \in S/\mu$ and $[a]\perp [b]$. Then $[0] = [a][b] = [ab]$. So $ab \mathrel{\mu} 0$,
thus ${\bf d}(ab) \mathrel{\mu} 0$, whence ${\bf d}(ab) = {\bf d}(0) = 0$ (as $\mu$ is idempotent-separating), so that $ab=0$. So $a' \perp b'$ for all $a' \in [a]$ and $b'\in [b]$. Let us check that $[a] \oplus [b]$ exists in $S/\mu$ and equals $[a\oplus b]$. We suppose that $[x] \geq [a],[b]$ and show that $[x]\geq [a\oplus b]$. We have that $[a] = [x]{\bf d}([a])=[x{\bf d}(a)]$ and similarly $[b] = [x{\bf d}(b)]$. By the definition of $\mu$ this means that for all $e\in E(S)$ we have ${\bf d}(ea) = {\bf d}(ex{\bf d}(a))$ and ${\bf d}(eb) = {\bf d}(ex{\bf d}(b))$. Then 
 ${\bf d}(e(a\oplus b)) = {\bf d}(ea \oplus eb) = {\bf d}(ea) \oplus {\bf d}(eb)$
and also ${\bf d}(ex{\bf d}(a\oplus b)) = {\bf d}(ex{\bf d}(a)\oplus ex{\bf d}(b)) = {\bf d}(ex{\bf d}(a)) \oplus  {\bf d}(ex{\bf d}(b))$. Therefore, $[a\oplus b] = [x{\bf d}(a\oplus b)]$, so that $[x]\geq [a\oplus b]$, as required. It follows that $S/\mu$ is a Boolean inverse semigroup and that $\mu^{\natural}$ preserves orthogonal joins.
\end{proof}

\subsection{Additive ideals and simple Boolean inverse semigroups}
An {\em additive ideal} of a Boolean inverse semigroup $S$ is a semigroup ideal which is closed with respect to binary orthogonal (or, equivalently, compatible) joins. An ideal $I$ of the generalized Boolean algebra $E(S)$ is called {\em closed with respect to conjugation} if $e\in I$ and $s\in S$ imply that $s^{-1}es\in I$. Additive ideals of $S$ and ideals of $E(S)$ are closely connected, as follows (see \cite[Proposition 3.4.8]{W17}).

\begin{proposition} \label{prop:ideals}\mbox{}
\begin{enumerate}
\item Let $I$ be an additive ideal of $S$. Then ${\bf d}(I) = \{{\bf d}(a)\colon a\in I\}$ is an ideal of the generalized Boolean algebra $E(S)$ which is closed with respect to conjugation.
\item Let $J$ be an ideal of $E(S)$ which is closed with respect to conjugation. Then $I(J) = \{s\in S\colon {\bf d}(s) \in J\}$ is an additive ideal of $S$.
\item The assignment $I\mapsto {\bf d}(I)$ is a bijection between additive ideals of $S$ and ideals of $E(S)$ which are closed with respect to conjutation.
\end{enumerate}
\end{proposition}

Similarly as in the ring theory, additive ideals give rise to congruences, as follows. For an additive ideal $I$ by $\varepsilon_I$ we denote the minimal congruence which identifies all the elements of $I$. If $c\in I$ and $a,b\perp c$ then necessarily $(a\oplus c) \mathrel{\varepsilon_I} (b\oplus c)$. On the other hand, it is easy to check that the relation $\sigma$, given by $s\mathrel {\sigma} t$ if and only if $s=a\oplus c$ and $t=b\oplus c$ for some $c\in I$, is a congruence on $S$ which identifies all the elements of $I$. Hence $\sigma=\varepsilon_I$. Observe that $I$ coinsides with the $\varepsilon_I$-class of $0$. This gives rise to the map $I\mapsto \varepsilon_I$ from additive ideals of $S$ to additive congruences on $S$. This map is clearly injective. As we shall now explain, it, however, is not in general surjective. 

If $\varphi\colon S\to T$ is a surjective additive homomorphism between Boolean inverse semigroups, $\varphi^{-1}(0)$ is an additive ideal of $S$, denote it by $I$. It  gives rise to the congruence $\varepsilon_I$ and we have $\varepsilon_I\subseteq {\mathrm{Ker}}\varphi$.  

Let $\sigma$ be a non-trivial idempotent-separating congruence. By Proposition \ref{prop:mu}, if $S$ is not fundamental, $\mu$ is an example of such a congruence. If $a\mathrel{\sigma} 0$, then ${\bf d}(a) \mathrel{\sigma} 0$ so that ${\bf d}(a) = 0$, and thus $a=0$. The $\sigma$-class of $0$ is thus $I=\{0\}$. The congruence $\varepsilon_{\{0\}}$ is of course the trivial congruence and it does not coincide with $\sigma$. It follows that the kernel of $\sigma$ is not equal to any congruence of the form $\varepsilon_I$.  We say that a congruence is {\em ideal-induced} if it coincides with $\varepsilon_I$ for some $I$. It follows that, in contrast to the ring theory, it is in general not true that any congruence on a Boolean inverse semigroup is ideal-induced. 

\begin{proposition} \label{prop:factorization} Let $\sigma$ be an additive congruence on a Boolean inverse semigoroup $S$. Then $S/\sigma$ is a quotient of $S/\varepsilon_I$, where $I$ is the $\sigma$-class of $0$, by an idempotent separating congruence. 
\end{proposition}

\begin{proof} The first claim holds since $\varepsilon_I \subseteq \sigma$. For the second one, observe that the inverse image of $0$ under the canonical map $S/\varepsilon_I\to S/\sigma$ is just $0$, so that if idempotents $e$ and $f$ are identified by this canonical map, then $e$ is identified with $ef$, so that $e\setminus ef$ is identified with $0$. Thus $e=ef$ and similarly $f=ef$ which yields the claim.
\end{proof}

\begin{definition} {\em
A Boolean inverse semigroup $S$ is called:
\begin{enumerate}
\item {\em simple}, if it does not have any proper non-trivial additive congruences or, equivalently, if every homomorphic image of $S$ (under an additive homomorphism between Boolean inverse semigroups) is either $0$ or isomorphic to $S$;
\item {\em additively 0-simple}, if it has no non-zero proper additive ideals.
\end{enumerate}}
\end{definition}

\begin{theorem} \cite[Theorem 2.7]{StSz} A Boolean inverse semigroup $S$ is simple if and only if it is fundamental and additively $0$-simple.
\end{theorem}

\begin{proof} The `if' direction is clear: if $S$ is not fundamental, then $\mu$ is a non-trivial proper congruence, if it has a proper non-trivial ideal $I$ then $\varepsilon_I$ is such a congruence. Conversely, suppose that $S$ is fundamental and additively $0$-simple and let $\sigma$ be an additive congruence on $S$. Then $[0]$ is an additive ideal, so it is $\{0\}$ or $S$. If it is $S$ then $\sigma$ is the universal congruence, and we are done. Suppose $[0] = \{0\}$. A similar argument as in the proof of Proposition \ref{prop:factorization} (or an application of this proposition) yields that $\sigma$ 
is idempotent-separating, so it is a trivial congruence, as $S$ is fundamental.
\end{proof}

\section{Semisimple and finite Boolean inverse semigroups} \label{s:semisimple} The results in this section are a special case of the non-commutative Stone duality for Boolean inverse semigroups \cite{KL17, LL13},  see \cite{L21}. Here we provide direct elementary proofs not invoking the non-commutative Stone duality.  The construction should be thought of as a non-commutative extension of the correspondence between finite sets and finite Boolean algebras.

\subsection{Semisimple Boolean inverse semigroups}
\begin{definition} {\em (Semisimple Boolean inverse semigroups \cite[Definition 3.7.5]{W17}) A Boolean inverse semigroup $S$ is called {\em semisimple}, if for each $s\in S$ the set $s^{\downarrow} = \{t\in S\colon t\leq s\}$ is finite.}
\end{definition}

A non-zero element $s\in S$ is called an {\em atom} if it is $0$-minimal with respect to the natural partial order on $S$, that is, if $0< t\leq s$ implies $t=s$. Note that any two distinct compatible atoms are necessarily orthogonal. 

Suppose that $S$ is a semisimple Boolean inverse semigroup and $s\in S\setminus\{0\}$.  Then every chain of elements between $s$ and $0$ is finite and thus contains an atom and there are only finitely many atoms below $s$. The join of these atoms, $t$, is below $s$ as well. If $t\neq s$, then $s\setminus t \neq 0$, and there is an atom $u\leq s\setminus t$.
But then $u\leq s$, so that $u\leq t$, by the definition of $t$. It follows that $t=s$. So every non-zero element $s\in S$ is the join of all the atoms below it. Moreover,  two distinct atoms below $s$ are necessarily orthogonal. For $s,t\in S$ let $s_1,\dots, s_n$ and $t_1,\dots, t_k$ be the lists of atoms below $s$ and $t$, respectively. Then
$$
st = (s_1\oplus \cdots \oplus s_n)(t_1\oplus \dots \oplus t_k) = \bigoplus \{s_it_j \colon 1\leq i\leq n, 1\leq j\leq k\}.
$$

Note that every $s_it_j$ is either zero or an atom. Indeed, suppose $u\leq s_it_j$. Then $u=s_it_je$ for some $e\in E(S)$. But $t_je\leq t_j$, so $t_je$ is either $0$ or $t_j$. In the first case $u=0$, and in the second case $u=s_it_j$ and it is an atom. As $s_i\leq s$ and $t_j\leq t$, we have $s_it_j\leq st$. Conversely, any atom $u$ below $st$ can be written as
$u=ste$ for some $e\in E(S)$ which rewrites to $u=(s{\bf r}(te))(te)$. Observe that both $s{\bf r}(te)\leq s$ and $te\leq t$ must be atoms, as otherwise we would find a non-zero element properly below $u$. Hence the non-zero products $s_it_j$ in fact run through the set of all the atoms below $st$. It follows that the multiplication in $S$ can be recovered from the multiplication of atoms using orthogonal joins. Note that if $u$ is an atom then so are ${\bf d}(u)$ and ${\bf r}(u)$. Moreover, idempotent atoms coincide with the atoms of $E(S)$. 

It follows that the atoms of $S$ form a groupoid, called the  {\em groupoid of atoms} of $S$, denote it by  ${\mathcal G}_{atoms}(S)$. Its objects are the atoms which are idempotents, and the arrows from $e$ to $f$ are the atoms $s$ with ${\bf d}(s) = e$ and ${\bf r}(s)=f$. Orthogonal finite sets of atoms in $S$ are in bijection with local bisections of ${\mathcal G}_{atoms}(S)$ where a {\em local bisection} of a groupoid is its subset, $A$, such that the restrictions of the domain and the range maps of the groupoid to $A$ are injective. We arrive at the following statement.

\begin{proposition} \label{prop:atoms} Any semisimple Boolean inverse semigroup is isomorphic to the inverse semigroup of all finite local bisections of its groupoid of atoms.
\end{proposition}

For a groupoid ${\mathcal G}$ we denote its inverse semigroup of all finite local bisections by $I({\mathcal G})$. Proposition \ref{prop:atoms} and the discussion that preceeds it establish an isomorphism between a semisimple Boolean inverse semigroup $S$ and the inverse semigroup $I({\mathcal G}_{atoms}(S))$.

Recall that a groupoid is called {\em principal} if all its isotropy groups (that is, the groups consisting of all elements of the groupoid whose domain and range coincide). 

\begin{proposition} \label{prop:structure} Let $S$ be a semisimple Boolean inverse semigroup.
\begin{enumerate}
\item $S$ is fundamental if and only if  ${\mathcal G}_{atoms}(S)$ is principal.
\item Additive ideals of $S$ are in a bijection with connected components of  ${\mathcal G}_{atoms}(S)$ via the map which assigns to a connected component $\Gamma$ of ${\mathcal G}_{atoms}(S)$ the ideal consisting of all finite local bisections contained in $\Gamma$.
\end{enumerate}
\end{proposition}

\begin{proof} (1) The statement basically follows from Proposition \ref{prop:atoms},  Proposition \ref{prop:fund} and the observation that an atom $s$ commutes with all the idempotents if and only if ${\bf d}(s) = {\bf r}(s)$. This follows from the fact that $s{\bf d}(s)=s$ and $${\bf d}(s)s = \left\lbrace \begin{array}{ll} s, & \text{if } {\bf d}(s) = {\bf r}(s),\\ 0, & \text{otherwise.}\end{array}\right.$$
(2) The assignment is clearly well defined and injective.  By Proposition \ref{prop:ideals} it is also surjective.
\end{proof}

We now recall the well known structure of connected groupoids. Let $X$ be a set and $G$ a group. Then $X\times G\times X$ is a connected groupoid with objects $X$ and maps from $x$ to $y$ being the elements $(y,g,x)$ where $g$ runs through $G$. The multiplication is given by $(z,h,y)(y,g,x) = (z,hg,x)$ and $(y,g,x)^{-1} = (x,g^{-1},y)$. Furthermore, every connected (discrete) groupoid is of this form. The groupoid $X\times G\times X$ is principal if and only if $G$ is the trivial group, in which case the groupoid is just (isomorphic to) the pair groupoid $X\times X$. The finite local bisections of such groupoids are given by rook matrices (that is, matrices with at most one non-zero element in each row and column\footnote{The term stems from thinking of such a matrix as a chess board, and of the non-zero elements of the matrix as rooks which do not attack each other.}) with rows and columns indexed by $X$ with finitely many non-zero entries from $G^0 = G\cup \{0\}$ (where $0\not\in G$): a finite local bisection $s\subseteq X\times G\times X$ gives rise to a matrix $A_s = (a_{ij})_{i,j\in X}$ where $a_{ij} = \left\lbrace\begin{array}{ll} g, & \text{if } (j,g,i)\in s,\\ 0, & \text{otherwise.} \end{array}\right.$ It follows that $I(X\times G\times X)$ is isomorphic to the semigroup $M_{fin}(X,G^0)$ of all rook matrices with rows and columns indexed by $X$ over $G^0$ with finitely many non-zero entries. In the case where $G = \{1\}$ (that is, when the groupoid $X\times G\times X$ is principal), we get the Boolean inverse semigroup  $M_{fin}(X,\{0,1\})$ which is isomorphic to the {\em finitary symmetric inverse semigroup} ${\mathcal I}_{fin}(X)$ which consists of all injective maps $Y\to X$ where $Y$ is a finite subset of $X$. If $X$ is finite, $M_{fin}(X,\{0,1\})$ is isomorphic to ${\mathcal I}(X)$.

In the case when there is no restriction on the number of connected components of ${\mathcal G}$, we get that $I({\mathcal G})$ is isomorphic to the {\em direct sum} $\bigoplus_{i\in I}M_{fin}(X_i,G_i^0)$ where $i$ runs over the set $I$ of connected components of the groupoid.

\subsection{Finite Boolean inverse semigroups.} Any finite Boolean inverse semigroup is semisimple, moreover, finite Boolean inverse semigroups are precisely those semisimple Boolean inverse semigroups $S$ for which ${\mathcal G}_{atoms}(S)$ is finite. It follows that finite semisimple Boolean inverse semigroups are precisely of the form $I({\mathcal G})$ where ${\mathcal G}$ is finite. It follows from the discussion above that any finite Boolean inverse semigroup is isomorphic to the finite direct sum $M_{n_1}(G_1^{0})\oplus \cdots \oplus M_{n_k}(G_k^0)$ where $G_1,\dots, G_n$ are finite groups (where $M_n(G^0)$ is the inverse semigroup of all $n\times n$ rook matrices over $G^0$). We obtain the following. 

\begin{proposition}\label{prop:finite_fundamental} Let $S$ be a finite Boolean inverse semigroup. Then $S$ is fundamental if and only if it is isomorphic to a finite direct sum ${\mathcal I}_{n_1} \oplus \cdots \oplus {\mathcal I}_{n_k}$ of symmetric inverse semigroups. It is fundamental and additively $0$-simple if and only if it is isomorphic to a finite symmetric inverse semigroup ${\mathcal I}_n$.
\end{proposition}

\begin{remark} \label{rem:duality} {\em Generalizing the presented ideas, one can assign to an arbitrary Boolean inverse semigroup its {\em Stone groupoid}, which is an \'etale topological groupoid ${\mathcal G} = ({\mathcal G}^{(0)}, {\mathcal G}^{(1)})$ where the topology of the unit space ${\mathcal G}^{(0)}$ is the usual Stone topology of the dual space of the generalized Boolean algebra $E(S)$, and ${\mathcal G}^{(1)}$ is the set of ultrafilters of $S$. Moreover, this assignment gives rise to a categorical duality between Boolean inverse semigroups and Stone groupoids, see Lawson \cite{L10} (and also also Lawson \cite{L12} and Lawson and Lenz \cite{LL13}).  A precursor of this duality is due Resende \cite{R07}, who established a bijection between pseudogroups (that is, complete and infinitely distributive inverse semigroups) and localic \'etale groupoids. The approaches by Lawson and Lenz, on the one hand, and Resende, on the other one, were unified and extended, from inverse to restriction semigroups, and from \'etale groupoids to \'etale categories, by Kudryavtseva and Lawson \cite{KL17}. Recently, de Castro and Machado \cite{CM23} defined and studied non self-adjoint operator algebras associated to the objects of \cite{KL17} which naturally generalize groupoid $C^*$-algebras.  We mention also the work \cite{CG} by Cockett and Garner, where a variant of the mentioned dualities is developed from a general category-theoretical perspective, with inverse or restriction semigroups replaced by (left) join-restriction categories with local glueings.}
\end{remark}

\section{The type monoid of a Boolean inverse semigroup}\label{s:type}
\subsection{Partial commutative monoids} 
The following is \cite[Defition 2.1.1]{W17}.
\begin{definition} {\em (Partial commutative monoid) \label{def:partial_com} A {\em partial commutative monoid} is a structure $(P, \oplus, 0)$ where $\oplus$ is a partially defined binary operation on $P$ and the following axioms hold:
\begin{enumerate}
\item[(Assoc)] $(p\oplus q) \oplus r$ is defined if and only if $p\oplus (q \oplus r)$ is defined, and in the latter case $(p\oplus q) \oplus r = p\oplus (q \oplus r)$, for all $p,q,r\in P$.
\item[(Com)]  $p\oplus q$ is defined if and only if $q\oplus p$ is defined, and in the latter case $p\oplus q= q\oplus p$, for all $p,q\in P$.
\item[(Zero)] $0\oplus p$ is defined and equals $p$, for all $p\in P$.
\end{enumerate}
A partial commutative monoid $(P,\oplus, 0)$ is called {\em conical}, if it satisfies:
\begin{enumerate}
\item [(Conical)] if $0=p\oplus q$ for some $p,q\in P$ then $p=0$ (and thus, by (Com), also $q=0$).
\end{enumerate}
It is called {\em cancellable} if it satisfies:
\begin{enumerate}
\item[(Cancel)] If $p\oplus q = p\oplus r$ then $q=r$.
\end{enumerate}}
\end{definition}

In what follows, we restrict our attention only to conical partial commutative monoids. 

The {\em algebraic preordering} $\leq$ on a partial commutative monoid $P$ is given by $x\leq y$ where $x,y\in P$, if there is $z\in P$ such that $y=x\oplus z$. If $S$ is a conical cancellative partial commutative monoid, then $\leq$ is a partial order.

\begin{definition} {\em (Homomorphisms and $V$-homomorphisms)
A map $f\colon P\to Q$ between partial commutative monoids is said to be:
\begin{itemize}
\item  a {\em homomorphism}, if $f(0) = 0$ and $f(x\oplus y) = f(x)\oplus f(y)$ for all $x,y\in P$ such that $x\oplus y$ is defined;
\item a  $V$-{\em homomorphism}, if it is a homomorphism and if $f(x) = y\oplus z$ for some $x\in P$ and $y,z\in Q$ then there are $s,t\in P$ such that $x=s\oplus t$ and $f(s)=y$, $f(t) = z$;
\item {\em conical}, if $f(x)=0$ implies $x=0$.
\end{itemize}}
\end{definition}

Let $P, M$ be partial commutative monoids and $P\subseteq M$. We call $P$ a {\em lower interval} of $M$ if the inclusion map $\iota$ of $P$ into $M$ is a $V$-homomorphism. Note that lower invervals are precisely lower subsets with respect to $\leq$ with the partial addition inherited from $M$. Indeed, if $P$ is a lower interval of $M$ and $p\in P$, $m\in M$ are such that $m\leq p$ then $p=m\oplus n$ in $M$ for some $n\in M$. As $\iota$ is a $V$-homomorphism, this gives $p=q\oplus r$ for some $q,r\in P$ with $\iota(p)=m$ and $\iota(q)=n$. So $p=m$ and $q=n$ and in particular $m\in P$. Conversely, a lower subset with respect to $\leq$ is clearly a lower interval.

\subsection{The enveloping commutative monoid.} From now till the end of this section by $S$ we denote a conical partial commutative monoid. We now define a commutative monoid $U(S)$  which contains $S$ so that the inclusion map is a $V$-homomorphism and which is the `universal envelope' of $S$. Let $\overline{S} = \{s\colon s\in S\}$ be a disjoint copy of the underlying set of $S$. We define $U(S)$ to be the commutative monoid, freely generated by $\overline{S}$ subject to the following relations:
\begin{itemize}
\item[(U1)] $\overline{0} = 0$,
\item[(U2)] $\overline{p} + \overline{q}  = \overline{r}$ whenever $p  \oplus q$ is defined in $S$ and equals $r$.
\end{itemize}

Condition (U1) means that $\overline{0}+\overline{p} = \overline{p}$ for all $p\in S$. The following is a special case of \cite[Propositions 2.1.7, 2.1.8]{W17}.

\begin{proposition} \mbox{}\label{prop:univ}
\begin{enumerate}
\item The map $\iota\colon S\to U(S)$, $s\mapsto \overline{s}$ is a $V$-embedding (that is, an injective $V$-homo\-mo\-rphism).
\item Let $T$ be a commutative monoid (with the unit $0$) and $f\colon S\to T$ be a homomorphism. Then there is a unique monoid homomorphism $g\colon U(S)\to T$ such that $g\iota = f$.
\end{enumerate}
\end{proposition}

\begin{proof} Elements of the free monoid over $\overline{S}$ are finite (possibly empty) sums $\sum k_s\overline{s}$. By (U1),  all elements $n\overline{0}$ are identified with $0$ in $U(S)$. Let us show that if $\overline{s} = \sum k_t \overline{t}$, then $s = \oplus k_t t$ in $S$. We prove this statement by induction. At the first step, we replace $\overline{s}$ either by $\overline{s} + \overline{0}$ or by $\overline{a} + \overline{b}$ where $a\oplus b$ is defined in $S$. So the needed statement holds. For an induction step, suppose that the statement holds after several applications of the defining relations to $\overline{s}$, so that $\overline{s}$ is written as $\sum k_t \overline{t}$ where $s = \oplus k_t t$. As the next step, we have one of the following possibilities:
\begin{enumerate}[(a)]
\item Suppose that we have either removed or added one summand $\overline{0}$. Then the statement holds as $s\oplus 0$ is defined for all $s\in S$. 
\item Suppose that we have replaced some $\overline{p} + \overline{q}$ by $\overline{r}$ where $p\oplus q = r$ in $S$. Using (Assoc), the statement still holds.
\item Suppose that we have replaced some $\overline{r}$ by $\overline{p} + \overline{q}$ where $p\oplus q = r$ in $S$. The statement holds again applying (Assoc). 
\end{enumerate}
It follows that $\overline{s} = \overline{t}$ implies $s=t$ in $S$, so that $\iota$ is injective. Let $\overline{s} = u + v$ in $U(S)$. As we have proved then $u = \sum k_s s$ and $v=\sum n_t t$ where
$(\oplus k_s s)\oplus (\oplus n_t t)$ is defined $S$. In particular, $x=\oplus k_s s$ and $y = \oplus n_t t$ are defined in $S$. It follows that $u=\iota(x)$ and $v=\iota(y)$, so that $\iota$ is a $V$-homomorphism.

(2) This is immediate.
\end{proof}

Proposition \ref{prop:univ}(1) says that the canonical inclusion $\iota$ identifies $S$ with a lower interval in $(U(S),\leq)$. 

According to \cite[Definition 2.2.1]{W17}, a partial commutative monoid $(P,\oplus, 0)$ is called a {\em partial refinement monoid} if it satisfies the {\em refinement property}: for all $a,b,c,d\in P$ satisfying $a\oplus b=c\oplus d$ there are $e_{11}, e_{12}, e_{21}, e_{22} \in P$ such that
$$
a=e_{11}\oplus e_{12}, \, b=e_{21}\oplus e_{22}, \, c=e_{11}\oplus e_{21}, \, d=e_{12}\oplus e_{22}.
$$
That is, elements $a$ and $b$ are the sums of the rows and $c$ and $d$ the sums of the columns of the following table:

\begin{center}
\begin{tabular}{|c | c| c|c|} 
 \hline
 & $c$ & $d$\\   
 \hline
$a$ & $e_{11}$ & $e_{12}$ \\ 
 \hline
$b$ & $e_{21}$ & $e_{22}$  \\ [1ex] 
 \hline
\end{tabular}
\end{center}

If $P$ is a monoid (that is, if the partial addition $\oplus$ is always defined) and satisfies the refinement property, it is called a {\em refinement monoid.}
We highlight the following interesting fact \cite[Proposition 2.2.4]{W17}.

\begin{proposition}  \label{prop:only}
Suppose that $U(S)$ is a refinement monoid (which requires $S$ be a partial refinement monoid). Then $U(S)$ is, up to isomorphism, the only commutative monoid $M$, which contains (an isomorphic copy of) $S$, is generated by $S$ and is such that the inclusion map of $S$ into $M$ is a $V$-embedding (or, equivalently, the only commutative monoid $M$, which contains $S$ as a lower interval).
\end{proposition}

\subsection{Green's relation ${\mathcal D}$} Let $S$ be a semigroup and $a,b\in S$.
By definition, we have $a \mathrel{\mathcal L} b$ if $a$ and $b$ generate the same principal left ideal, that is, if $S^1a = S^1b$ where $S^1$ equals $S$, if $S$ has an identity element, and $S\cup \{1\}$ where $1\not\in S$ is an external identity element, otherwise. Dually, $a \mathrel{\mathcal R} b$ holds provided that $aS^1 = bS^1$. It is immediate that ${\mathcal L}$ and ${\mathcal R}$ are equivalence relations. Define $a\mathrel{\mathcal D} b$ if there is $c\in S$ such that $a \mathrel{\mathcal L} c$ and $c\mathrel{\mathcal R} b$, that is, ${\mathcal D}={\mathcal L} \circ {\mathcal R}$. It is well known  \cite[Proposition 2.1.3]{H} and easy to show that the relations ${\mathcal L}$ and ${\mathcal R}$ commute, and so ${\mathcal D} = {\mathcal L}\circ {\mathcal R} = {\mathcal R}\circ {\mathcal L}$. It follows that ${\mathcal D}$ is itself an equivalence relation and is the smallest equivalence relation containing ${\mathcal L}$ and ${\mathcal R}$. 

\subsection{The set ${\mathrm{Int}}(S)$} Let $S$ be an inverse semigroup and $s\in S$. Since $s= ss^{-1}s\in ss^{-1}S^1$ and $ss^{-1}\in sS^1$, we have $s \mathrel{\mathcal R} {\bf r}(s)$. Similarly, $s \mathrel{\mathcal L} {\bf d}(s)$. Since in an inverse semigroup every ${\mathcal L}$-class and every ${\mathcal R}$-class contains a unique idempotent \cite[Theorem 5.1.1]{H}, we have the following.

\begin{lemma} \label{lem:conc1} Let $s\in S$. Then ${\bf d}(s)$ is the unique idempotent which is ${\mathcal L}$-related with $s$ and ${\bf r}(s)$ is the inique idempotent which is ${\mathcal R}$-related with $s$. 
\end{lemma}

We will be interested in the restriction of the Green's relation ${\mathcal D}$ to $E(S)$. 

\begin{lemma} \label{lem:charD} Let $e,f\in E(S)$. The following conditions are equivalent:
\begin{enumerate}
\item $e\mathrel{\mathcal D} f$,
\item there is $s\in S$ satisfying ${\bf d}(s) = e$ and ${\bf r}(s) = f$,
\item there is $s\in S$ satisfying ${\bf r}(s) = e$ and ${\bf d}(s) = f$,
\item there is $s\in S$ satisfying $s^{-1}fs = e$ and $ses^{-1}=f$.
\end{enumerate}
\end{lemma}

\begin{proof} If $e\mathrel{\mathcal D} f$, there is $s\in S$ such that $e\mathrel{\mathcal L} s \mathrel{\mathcal R} f$.
Lemma \ref{lem:conc1} implies that $e={\bf d}(s)$ and $f= {\bf r}(s)$. In addition, ${\bf d}(s) \mathrel{\mathcal{L}} s \mathrel{\mathcal R}{\bf r}(s)$ implies ${\bf d}(s) \mathrel{\mathcal D} {\bf r}(s)$. It follows that (1) $\Leftrightarrow$ (2). Conditions (2) and (3) are equivalent, too, since ${\bf d}(s) = {\bf r}(s^{-1})$ and ${\bf r}(s) = {\bf d}(s^{-1})$. To show (1) $\Rightarrow$ 
(4), we let $e\mathrel{\mathcal D} f$ and take $s$ with $e\mathrel{\mathcal L} s \mathrel{\mathcal R} f$. Multiplying $f=ss^{-1}$ from the left by $s^{-1}$ and from the right by $s$ we obtain $s^{-1}fs = e$ and similarly $ses^{-1}=f$. Finally, if these equalities hold, we have $f=sees^{-1} = {\mathbf r}(se)$ and ${\bf d}(se) = se(se)^{-1} = sees^{-1} = ses^{-1}=f$, and we have proved (4) $\Rightarrow$ (3).
\end{proof}

\subsection{The type monoid ${\mathrm{Typ}}(S)$}

Let $S$ be a Boolean inverse semigroup.  For $e\in E(S)$ let $[e]$ be the intersection of the ${\mathcal D}$-class of $e$ with $E(S)$ or, equivalently, the ${\mathcal D}|_{E(S)}$-class of $e$. Put differently,
$$[e] = \{f\in E(S)\colon \text{there is } s\in S \text{ with } {\bf d}(s) = e \text{ and } {\bf r}(s) = f\}.$$
Following \cite{W17}, the quotient set $E(S)/{\mathcal D}|_{E(S)}=\{[e] \colon e\in E(S)\}$ will be denoted by ${\mathrm{Int}}(S)$. 

We now define the partial operation $\oplus$ on ${\mathrm{Int}}(S)$ by letting the sum $[e] \oplus [f]$ be defined, if there are $e' \in [e]$ and $f' \in [f]$ such that $e'\perp f'$, in which case we put $[e]\oplus [f] = [e'\oplus f']$. We first check this is well defined.

\begin{lemma} If $e',e'' \in [e]$ and $f',f'' \in [f]$ are such that $e' \perp f'$ and $e'' \perp f''$ then $[e''\oplus f''] = [e'\oplus f']$. 
\end{lemma}

\begin{proof} Since $e' \mathrel{\mathcal D} e''$, there is $s\in S$ with ${\bf d}(s) = e'$ and ${\bf r}(s) = e''$, by Lemma \ref{lem:charD}. Similarly, there is 
$t\in S$ with ${\bf d}(t) = f'$ and ${\bf r}(t) = f''$. By Lemma \ref{lem:ort}(3) it follows that $s\perp t$. Hence there is $s\oplus t$. Moreover, ${\bf d}(s\oplus t) = {\bf d}(s) \oplus {\bf d}(t) = e'\oplus f'$ and ${\bf r}(s\oplus t) = {\bf r}(s) \oplus {\bf r}(t) = e''\oplus f''$ by Proposition \ref{prop:join}. Hence $[e' \oplus f'] = [e''\oplus f'']$, again  by Lemma \ref{lem:charD}. 
\end{proof}
 
If $S$ is a Boolean inverse semigroup, ${\mathrm{Int}}(S)$ is a partial commutative monoid (see \cite[Proposition 3.3]{KLLR}). It was introduced by Wallis \cite{Wal}. Moreover, ${\mathrm{Int}}(S)$ is conical and is a partial refinement monoid. 

We now define the {\em type monoid} ${\mathrm{Typ}}(S)$ of $S$. Let $\overline{{\mathrm{Int}}(S)} = \{\overline{e} \colon  [e]\in {\mathrm{Int}}(S)\}$ be a disjoint copy of the underlying set of ${\mathrm{Int}}(S)$. Elements of $\overline{{\mathrm{Int}}(S)}$ are thus of the form $\overline{e}$ with $e\in E(S)$ where $\overline{e} = \overline{f}$ whenever $e \mathrel{\mathcal{D}} f$.
We define ${\mathrm{Typ}}(S)$ to be the commutative monoid freely generated by  $\overline{{\mathrm{Int}}(S)}$ subject to the following relations:
\begin{enumerate}
\item[(T1)] $\overline{0}  = 0$,
\item[(T2)] $\overline{e} + \overline{f}  = \overline{g}$ whenever $[e]  \oplus [f]$ is defined in ${\mathrm{Int}}(S)$ and equals $[g]$.
\end{enumerate}

There is a map $\iota\colon {\mathrm{Int}}(S) \to {\mathrm{Typ}}(S)$ given by $[e] \mapsto \overline{e}$. By \cite[Corollary 4.1.4]{W17} ${\mathrm{Typ}}(S)$ is a
conical refinement monoid, and ${\mathrm{Int}}(S)$ is a lower interval of ${\mathrm{Typ}}(S)$. By Proposition~\ref{prop:only}, ${\mathrm{Typ}}(S)$ is the only commutative monoid which is generated by ${\mathrm{Int}}(S)$ and contains it as a lower interval.

\begin{example}\label{ex:typ} {\em Let us calculate the type monoid of ${\mathcal I}_n$. For idempotents $e,f$ we have $e \mathrel{{\mathcal D}} f$ if and only if they are of the same rank, so ${\mathrm{Int}}({\mathcal I}_n) = \{0,1,2,\dots, n\}$ where $k$ stands for the ${\mathcal D}|_{E({\mathcal I}_n)}$-class consisting of the idempotents of rank $k$. Observe that $k\oplus m$ is defined if and only if $k+m\leq n$. Since there are $n$ pairwise orthogonal idempotents of rank $1$, $k=\oplus_{i=1}^k 1$ for all $k=1,\dots, n$. It follows that ${\mathrm{Int}}({\mathcal I}_n)$ is generated by $1$ and, since there are no identities, the monoid ${\mathrm{Typ}}({\mathcal I}_n)$ is freely generated by $1$ and is thus isomorphic to ${\mathbb N}_{0}$.}
\end{example}
A Boolean inverse semigroup is called {\em orthogonally separating} \cite{KLLR} if the partial addition $\oplus$ is everywhere defined, so that ${\mathrm{Int}}(S)={\mathrm{Typ}}(S)$ is a monoid.

\subsection{Generalized rook matrices over a Boolean inverse semigroup}\label{subs:rook}
Let $S$ be a Boolean inverse semigroup. By $M_{\omega}(S)$ we denote the set of all matrices over $S$, called {\em generalized rook matrices}, with rows and columns indexed by ${\mathbb N}$, which have only finitely many non-zero entries satisfying the following conditions:
\begin{enumerate}
\item If $a$ and $b$ are in the same row, then ${\bf r}(a)\perp {\bf r}(b)$.
\item If $a$ and $b$ are in the same column, then ${\bf d}(a)\perp {\bf d}(b)$.
\end{enumerate}

We  can define the multiplication on $M_{\omega}(S)$ as the usual multiplication of finitary infinite matrices with $+$ replaced by $\vee$, as follows: if $A=(a_{ij})_{i,j\in {\mathbb N}}$ and  $B=(b_{ij})_{i,j\in {\mathbb N}}$ define $AB=C = (c_{ij})_{i,j\in {\mathbb N}}$ where $c_{ij}=0$ if the $i$-th row of $A$ or the $j$-th column of $B$ are zero, or, otherwise, $c_{ij}=\bigvee\limits_{k} a_{ik}b_{kj}$ where the join is over all elements $k\in {\mathbb N}$ satisfying $a_{ik}\neq 0$ or $b_{kj}\neq 0$. Such a join is necessarily finite.

\begin{lemma}  $AB$ is well defined and belongs to $M_{\omega}(S)$.
\end{lemma}

\begin{proof} It is enough to suppose that the $i$-th row of $A$ and the $j$-th column of $B$ are non-zero and prove that the elements $a_{ik}b_{kj}$ are pairwise compatible, so their join is well defined. We prove that these elements are in fact pairwise orthogonal. Lemma \ref{lem:dom1} implies that ${\bf d}(a_{ik}b_{kj}) \leq {\bf d}(b_{kj})$. Then, for $k\neq t$, we have ${\bf d}(a_{ik}b_{kj}){\bf d}(a_{it}b_{tj})\leq {\bf d}(b_{kj}){\bf d}(b_{tj}) = 0$. Similarly, for $k\neq t$, ${\bf r}(a_{ik}b_{kj}){\bf r}(a_{it}b_{tj}) = 0$. Lemma \ref{lem:ort}(3) now implies that $a_{ik}b_{kj}$ and $a_{it}b_{tj}$ (where $t\neq k$) are orthogonal, and thus compatible, which completes the proof.
\end{proof}
 
\begin{proposition}[\cite{KLLR, Wal}] Let $S$ be a Boolean inverse semigroup.
\begin{enumerate}
\item $M_{\omega}(S)$ with the operation of the product of generalized rook matrices is a semigroup.
\item Idempotents of $M_{\omega}(S)$ are diagonal matrices whose diagonal entries are idempotents.
\item If $A=(a_{ij})_{i,j\in{\mathbb N}}$ then $A^{-1} = (a_{ji}^{-1})_{i,j\in {\mathbb N}}\in M_{\omega}(S)$. In addition, $A=AA^{-1}A$ and $A^{-1} = A^{-1}A A^{-1}$.
\item $M_{\omega}(S)$ is an inverse semigroup where $A^{-1}$ is the inverse of $A$.
\item $M_{\omega}(S)$ is a Boolean inverse semigroup.
\item The natural partial order is given by $A\leq B$ if and only if $a_{ij}\leq b_{ij}$ for all $i,j$.
\end{enumerate}
\end{proposition}

By $\Delta(a_1,\dots, a_n)$ we denote the diagonal matrix in $M_{\omega}(S)$ whose first $n$ diagonal entries are $a_1,\dots, a_n$ and all the other entries are $0$.

\begin{lemma}\mbox{} \label{lem:D}
\begin{enumerate}
\item Let $A\in M_{\omega}(S)$ be a diagonal matrix with non-zero entries $e_i=a_{k_i,k_i}$, $1\leq i\leq n$. Then $A \mathrel{\mathcal D} \Delta(e_1,\dots, e_n)$.
\item Let $A,B\in M_{\omega}(S)$ be diagonal matrices with non-zero diagonal entries $e_1,\dots, e_n \in E(S)$. Then $A\mathrel{\mathcal D} B$.
\end{enumerate}
\end{lemma}

\begin{proof} (1) Let $B$ be the matrix with the only non-zero entries $b_{i,k_i} = e_i$, $1\leq i \leq n$. Then 
$B^{-1}B = A$ and $BB^{-1} = \Delta(e_1,\dots, e_n)$.

(2) This follows from (1) and transitivity of ${\mathcal D}$.
\end{proof}

\begin{lemma} \label{lem:d1} Let $e_i, e'_i \in E(S)$ and $e_i \mathrel{\mathcal D} e_i'$, $1\leq i\leq n$. Then $\Delta(e_1,\dots, e_n) \mathrel{\mathcal D} \Delta(e_1',\dots, e_n')$.
\end{lemma}

\begin{proof} Let $s_i\in S$ be such that ${\bf d}(s_i)=e_i$ and ${\bf r}(s_i) = e_i'$, $1\leq i\leq n$. For $A=\Delta(s_1,\dots s_n)$ we then have $A^{-1}A = \Delta(e_1,\dots, e_n)$ and $AA^{-1} = \Delta(e_1',\dots, e_n')$.
\end{proof}

\begin{lemma} \label{lem:d2} $M_{\omega}(S)$ is orthogonally separating and the addition in ${\mathrm{Typ}}(M_{\omega}(S)) = {\mathrm{Int}}(M_{\omega}(S))$ is given by $$[\Delta(e_1,\dots, e_n)] + [\Delta(f_1,\dots, f_k)] = [\Delta(e_1,\dots, e_n, f_1,\dots, f_k)].$$
\end{lemma}

\begin{proof} Let $A= \Delta(e_1,\dots, e_n)$ and $B = \Delta(f_1,\dots, f_k)$. By Lemma \ref{lem:D}(2) we have that 
$B \mathrel{\mathcal D} B'$ where $B' = \Delta(\underbrace{0,\dots, 0}_n, f_1,\dots, f_k)$. Since $A\perp B'$ and $A\oplus B' = \Delta(e_1,\dots, e_n, f_1,\dots, f_k)$, the statement follows.
\end{proof}

We now aim to give a presentation of ${\mathrm{Typ}(M_{\omega}(S)})$ by generators and relations. We put $X=E(S)/{\mathcal D}|_{E(S)}$. The elements of $X$ are thus $[e] = \{f\in E(S)\colon f\mathrel{\mathcal D} e\}$ where $e\in E(S)$. Let $F(X)$ be the free commutative monoid on $X$.

Let $A\in M_{\omega}(S)$ and suppose that the left upper $n\times n$ corner of $A$ is the matrix $A'$. If all non-zero entries of $A$ are in $A'$, we will sometimes write $A'$ for $A$, to simplify notation. Thus, by saying that $A'$ is in $M_{\omega}(S)$ (where $A'$ is an $n\times n$ matrix) we mean that $A'$ is the upper left corner of some $A\in M_{\omega}(S)$ all whose non-zero entries are in $A'$.

\begin{theorem}\label{th:type} ${\mathrm{Typ}(M_{\omega}(S)})$ is an $X$-generated commutative monoid via the map $[e] \mapsto [\Delta(e)]$ subject to the relations 
\begin{equation}\label{eq:def_rel}
[e] + [f] =[g]  \text{ whenever } [e] \oplus [f] = [g] \text{ in } {\mathrm{Int}}(S). 
\end{equation}
Consequently, the monoid ${\mathrm{Typ}}(S)$ is isomorphic to the monoid ${\mathrm{Typ}(M_{\omega}(S)}) = {\mathrm{Int}}(M_{\omega}(S))$.
\end{theorem}

\begin{proof} The map $[e] \mapsto [\Delta(e)]$ is well defined by Lemma \ref{lem:d1}. 

Let us show that relations \eqref{eq:def_rel} hold in ${\mathrm{Typ}}(M_{\omega}(S))$. That is, $[\Delta(e)] + [\Delta(f)] = [\Delta(e\oplus f)]$ whenever $e\perp f$ in $E(S)$.
By Lemma \ref{lem:d2} we have $[\Delta(e)] + [\Delta(f)] = [\Delta(e,f)]$ so it suffices to show that $\Delta(e,f) \mathrel{\mathcal D} \Delta(e\oplus f)$. Let $s,t \in S$ be such that ${\bf d}(s) = e$ and ${\bf d}(t) =f$. Then $A = \begin{bmatrix} s & 0 \\ t & 0\end{bmatrix} \in M_{\omega}(S)$ with $A^{-1} = \begin{bmatrix} s^{-1} & t^{-1} \\0 & 0\end{bmatrix}$. The statement follows, since $AA^{-1} = \Delta(e,f)$ and $A^{-1}A = \Delta(e\oplus f)$. Applying induction, it follows that $[\Delta(e_1)]+\cdots + [\Delta(e_n)] = [\Delta(\oplus_{i=1}^n e_i)]$ provided that the $e_i$'s are pariwise orthogonal idempoitents of $S$.

We finally show that any relation that holds in ${\mathrm{Typ}}(M_{\omega}(S))$ in fact follows from the relations \eqref{eq:def_rel}.
So suppose that $\sum k_{i}[e_i]$, $\sum n_{j}[f_j] \in F(X)$ are such that 
\begin{equation}\label{eq:aux1}
\sum k_i[\Delta(e_i)] =  \sum n_{j}[\Delta(f_j)] 
\end{equation}
in ${\mathrm{Typ}}(M_{\omega}(S))$. This means that $$E=\Delta(\underbrace{e_1,\dots, e_1}_{k_1}, \underbrace{e_2,\dots, e_2}_{k_2}, \dots) \mathrel{\mathcal D} \Delta(\underbrace{f_1,\dots, f_1}_{n_1}, \underbrace{f_1,\dots, f_1}_{n_2}, \dots) = F.$$ So there is a matrix $A\in M_{\omega}(S)$ with $A^{-1}A = E$ and $AA^{-1} = F$. Suppose that all the entries of $A$, but the upper left corner of size $n\times k$, are zeros. It follows that
$$
E= \Delta(\oplus_{i=1}^n{\bf d}(a_{i1}), \dots, \oplus_{i=1}^n {\bf d}(a_{ik})),
$$
$$
F = \Delta(\oplus_{j=1}^k {\bf r}(a_{1j}),\dots, \oplus_{j=1}^k {\bf r}(a_{nj})).
$$
So we can rewrite \eqref{eq:aux1} as
$$
\sum_{j=1}^k [\Delta(\oplus_{i=1}^n{\bf d}(a_{ij}))]  = \sum_{i=1}^n[\Delta(\oplus_{j=1}^k {\bf r}(a_{ij}))].
$$
Using \eqref{eq:def_rel} this rewrites to
$$
\sum_{i=1}^n\sum_{j=1}^k [\Delta({\bf d}(a_{ij}))] = \sum_{i=1}^n\sum_{j=1}^k [\Delta({\bf r}(a_{ij}))],
$$
which holds by Lemma \ref{lem:d1}, using the fact that ${\bf d}(a_{ij}) \mathrel{\mathcal D} {\bf r}(a_{ij})$ in $S$. This completes the proof.
\end{proof}

Note that the type monoid ${\mathrm{Typ}}(S)$ was defined in \cite{KLLR} as the commutative monoid ${\mathrm{Int}}(M_{\omega}(S))$. However, its universal property is not treated in \cite{KLLR}. Here we followed \cite{W17} and defined ${\mathrm{Typ}}(S)$ as the universal commutative monoid $U({\mathrm{Int}}(S))$. The statement of Theorem \ref{th:type} is proved in \cite[Proposition 4.2.5]{W17}, where it is derived from a more general result. Here we provided a different and direct proof. Let us briefly outline the more general setting treated in \cite{W17}. 

An additive homomorphism $f\colon S\to T$ of Boolean inverse semigroups gives rise to a homomorphism of partial monoids $\overline{f}\colon {\mathrm{Int}}(S) \to {\mathrm{Int}}(T)$, $\overline{f}([e]) = [f(e)]$. Compofsing $\overline{f}$ with $\iota_T\colon {\mathrm{Int}}(T)\to {\mathrm{Typ}}(T)$, we get the homomorphism $\iota_T\overline{f}\colon {\mathrm{Int}}(S)\to {\mathrm{Typ}}(T)$. Since ${\mathrm{Typ}}(S)$ is the additive envelope of ${\mathrm{Int}}(S)$, the universal property (Proposition \ref{prop:univ}(2)) implies that there is $g\colon {\mathrm{Typ}}(S) \to {\mathrm{Typ}}(T)$ such that $g\iota_S =  \iota_T\overline{f}$. We put $g={\mathrm{Typ}}(f)$. It follows that the assignment $S\mapsto {\mathrm{Typ}}(S)$ gives rise to a functor, denoted by ${\mathrm{Typ}}$, from the category of Boolean inverse semigroups and additive homomorphisms to the category of commutative monoids and homomorphisms. 
 
 Suppose that $\iota$ is the inclusion map of Boolean inverse semigroups from $S$ into $T$ and let us find a sufficient condition for $\overline{\iota}$ to be injective. If $\overline{\iota}([e]) = \overline{\iota}([g])$ then $e=\iota(e) \mathrel{\mathcal D} \iota(g)=g$ in $T$, so there is $s\in T$ with $e = {\bf d}(s)$ and $g={\bf r}(s)$. Then $s = gse \in STS$. If we assume that $STS \subseteq S$ (as $s=ss^{-1}s\in STS$, the opposite inclusion holds automatically), we obtain $s\in S$ and then $e \mathrel{\mathcal D} g$ in $S$ so that $\overline{\iota}$ is injective. An inverse subsemigroup $S$ of $T$ is called a {\em quasi-ideal} if $S=STS$. We have seen that if $S$ is a quasi-ideal of $T$ then $\overline{\iota}$ is injective. Furthermore, it is a $V$-homomorphism. Indeed, if $p\in E(S)$ and $[p] = [q]\oplus [r]$ in ${\mathrm{Int}}(T)$ then there are idempotents $p'\in [p], q' \in [q]$ and $r'\in [r]$ such that $q'\perp r'$ and $p'=q'\oplus r'$ in $T$. But then $q' = p'q'p' \in STS = S$ and similarly $r'\in S$. So $[p] = [q]\oplus [r]$ in ${\mathrm{Int}}(S)$, and $\overline{\iota}$ is a $V$-homomorphism.
 
 Now, if $S$ is a quasi-ideal of $T$, we know that $\overline{f}$ is injective, and \cite[Proposition 2.2.5]{W17} implies that ${\mathrm{Typ}}(f)$ is injective. Moreover, it is shown in \cite[Theorem 4.2.2]{W17} that ${\mathrm{Typ}}(f)$ is an isomorphism if and only if $T$ is an {\em additive enlargement} \cite[Proposition 3.1.20, Definition 3.1.21]{W17} of $S$, which means that the additive ideal, generated by $S$, coincides with $T$. 
 
Observe that, for a Boolean inverse semigroup $S$, there is an embedding $h\colon S \to M_{\omega}(S)$, given by $s\mapsto \Delta(s)$, and it is easy to check that $M_{\omega}(S)$ is an additive enlargement of $h(S)$. This leads to the statement of Theorem \ref{th:type}.

\section{The connection of ${\mathrm{Typ}}(S)$ with the monoid $V(K \langle S\rangle)$}\label{s:connection}

Let $R$ be a (not necessarily commutative) ring. Two idempotents $p,q\in R$ are said to be {\em Murray - von Neumann equivalent}, provided that there exist $x\in pRq$ and $y\in qRp$ such that $p=xy$ and $q=yx$. Let $M_n(R)$ be the ring of all $n\times n$ matrices over $R$ and consider the diagonal embeddings $A\mapsto \begin{bmatrix} A & 0 \\ 0 & 0\end{bmatrix}$ of $M_n(R)$ into $M_{n+1}(R)$. Let $M_{\infty}(R) = \lim\limits_{\longrightarrow} M_n(R)$ with respect to these embeddings. By $V(R)$ we denote the set of equivalence classes of idempotent matrices in  $M_{\infty}(R)$. If $[p], [q] \in V(R)$ we put $[p] + [q] = \left[\begin{bmatrix} p& 0 \\ 0 & q\end{bmatrix}\right]$. With this addition $V(R)$ is a commutative monoid, which encodes the {\em nonstable $K$-theory} of $R$ \cite{AP07,W17}. 
Alternatively, $V(R)$ can be described as the monoid of isomorphism classes of finitely generated right projective $R$-modules where the addition is given by the direct sum of representatives, see \cite{AP07} for details.

Let $K$ be a commutative ring with unit and $S$ a Boolean inverse semigroup. By $K\langle S\rangle$ we denote the quotient of the contracted semigroup algebra $K_0S$ by the ideal ${\mathcal T}_K(S)$ called the {\em tight ideal}  (see \cite{StSz}), generated by all $p-(q\oplus r)$ where $p,q,r\in S$, $q\perp r$ and $q\oplus r=p$. We note that the algebra $K\langle S\rangle$ is isomorphic to the algebra $K_0{\mathcal B}_{tight}(S)$ of the {\em tight Booleanization} ${\mathcal B}_{tight}(S)$ of $S$ (this follows, e.g., from \cite[Corollary 2.15]{StSz}).

\begin{example}\label{ex:in} {\em If $S = {\mathcal I}_n$ then $K\langle S\rangle$ is isomorphic to the full matrix algebra $M_n(K)$. Indeed, atoms of ${\mathcal I}_n$ are the rank one elements and are of the form $e_{ij}$ with ${\mathrm{dom}}(e_{ij}) = j$ and $e_{ij}(j) = i$ for all $1\leq i,j\leq n$. Since every element of ${\mathcal I}_n$ is an orthogonal join of atoms, the elements $e_{ij}$ form the basis of $K\langle S\rangle$. Mapping each $e_{ij}$ to a respective matrix unit leads to the desired isomorphism. }
\end{example}

Following \cite[Definition 6.3.2]{W17}, a map $f$ from a Boolean inverse semigroup $S$ to a ring $R$ is called an {\em additive measure} if it is a semigroup embedding of $S$ into a multiplicative semigroup of $R$ and $f(x\oplus y) = f(x) + f(y)$ whenever $x\perp y$ in $S$ (in particular, $f(0)=0$). Then the canonical map $S \to K \langle S\rangle$ is universal among all the additive measures from $S$ to a $K$-algebra. 

Let  $f\colon S\to R$ be an additive measure and suppose that $a,b\in E(S)$ are such that $a\mathrel{\mathcal D} b$. Then $a=s^{-1}s$ and $b=ss^{-1}$ for some $s\in S$, so the elements $f(a)$ and $f(b)$ are Murray-von Neumann equivalent in $R$. Hence for each $[e] \in E(S)/{\mathcal D|_{E(S)}}$ there is a well defined element $[f(e)]_R\in V(R)$. This map extends to the monoid homomorphism from the free commutative monoid on ${\mathrm{Int}}(S)$ to $V(R)$ and, due to the universality of ${\mathrm{Typ}}(S)$, to the monoid homomorphism ${\mathrm{Typ}}(S) \to V(R)$ which we denote by ${\bf f}$.

A Boolean inverse semigroup is called {\em matricial} if it is a finite direct product of finite symmetric inverse monoids, or, equivalently, if it is a finite fundamental Boolean inverse semigroup.
It is called {\em locally matricial} if it is a directed colimit of finite fundamental Boolean inverse semigroups.

\begin{proposition}\cite[Proposition 6.7.3]{W17} \label{prop:lm} Let $S$ be locally matricial. Then the canonical map ${\bf f}\colon {\mathrm{Typ}}(S) \to V(K \langle S\rangle)$ is an isomorphism.
\end{proposition}

\begin{proof} We make use of the fact that all the functors ${\mathrm{Typ}}$, $V$ and $K\langle\_\rangle$ preserve directed colimits and finite direct products. For the functor ${\mathrm{Typ}}$ this is proved in \cite[Proposition 4.1.9]{W17}, for the functor $V$ this is well known and for the functor $K\langle\_\rangle$ it is straightforward to show. 

So it suffices to consider the case of a finite fundamental Boolean inverse semigroup. As is shown in Proposition \ref{prop:finite_fundamental}, such a semigroup is a direct sum of finite symmetric inverse monoids, so it is enough to consider the case where $S$ is some ${\mathcal I}_n$. We have ${\mathrm{Typ}}({\mathcal I}_n) \simeq {\mathbb N}_{0}$, see Example \ref{ex:typ}. The statement follows, since also $V(M_n(K))$ is isomorphic to ${\mathbb N}_{0}$ and $K\langle {\mathcal I}_n\rangle \simeq M_n(K)$.
\end{proof}
Since locally matricial Boolean inverse semigroups include semisimple ones, the statement of Proposition \ref{prop:lm} holds, in particular, for semisimple and fundamental Boolean inverse semigroups (see \cite[Proof of Proposition 6.7.3]{W17}). Furthermore, if $S$ is semisimple (but not necessarily fundamental) one can show, using the fact that ${\mathrm{Typ}}(S) = {\mathrm{Typ}}({\mathrm{S/\mu}})$, that ${\bf f}$ is one-to-one (see \cite[Proposition 6.3.7]{W17} for details). It is interesting to note that ${\bf f}$ is not necessarily surjective, as is shown in \cite[Proposition 6.7.1]{W17} for the case where  $K$  is a division ring and  $S=G^{0}$ with $G$ being a group containing a torsion element of order which is not a power of the characteristic of $K$. The map ${\bf f}$ is not necessarily injective, either, see \cite[Lemma 6.7.6]{W17}.

\section{Type monoids of graph inverse semigroups}\label{s:graph}

\subsection{Graph inverse semigroups and Leavitt path algebras}
A directed graph $\Gamma$ consists of a vertex set $\Gamma_0$, an edge set $\Gamma_1$, and the source and range maps $s, r\colon \Gamma_1\to \Gamma_0$. We say that $\Gamma$ is a {\em row-finite graph} if each vertex in $\Gamma$ emits only a finite number of edges. If $e\in \Gamma_0$ is an edge, its {\em inverse edge}
$e^*$ satisfies $s(e^*) = r(e)$ and $r(e^*) = s(e)$. We also put $(e^*)^* = e$. Let $(\Gamma_1)^*$ denote the set $\{e^*\colon e \in \Gamma_1\}$ and assume that $\Gamma_1 \cap (\Gamma_1)^* = \varnothing$ and $e\mapsto e^*$ is a bijection from $\Gamma_1$ to $(\Gamma_1)^*$. For a path $p=e_1\dots e_n$ (with all $e_i\in {\Gamma^1}$) we put $s(p)=s(e_n)$ and $r(s) = r(e_1)$\footnote{We multiply edges from the right to the left, so that $e_n$ is the first edge of $p$, and $e_1$ is the last one.}, the inverse path of $p$ is denoted by $p^* = e_n^*\dots e_1^*$.
It is convenient to view the elements of $\Gamma_0$ as paths of length $0$, extending $s$ and $r$ to $\Gamma_0$ by $s(v) = r(v) = v$  and put $v^*=v$ for all vertices $v$.

The {\em graph inverse semigroup} $I(\Gamma)$ of a directed graph $\Gamma$ is the semigroup generated by $\Gamma_0 \cup \Gamma_1\cup \Gamma_1^*$ together with a zero $0$ subject to the relations:
\begin{enumerate}
\item[(G1)] $vv' = \delta_{v,v'}v$ for all $v,v'\in \Gamma_0$.
\item[(G2)] $es(e) = r(e)e =e$ for all $e\in \Gamma_1\cup \Gamma_1^*$.
\item[(G3)] $ee'^* = \delta_{e,e'}r(e)$ for all $e,e'\in  \Gamma_1$.
\end{enumerate}

If $K$ is a field, the {\em graph $K$-algebra} $L_K(\Gamma)$, or the {\em Leavitt path algebra} associated with $\Gamma$, is defined as the $K$-algebra, generated by $\Gamma_0 \cup \Gamma_1\cup \Gamma_1^*$ subject to the relations (G1)-(G3) and also the relation
\begin{enumerate}
\item[(G4)] $v = \sum\limits_{e\in \Gamma_1\colon s(e) = v} e^*e.$
\end{enumerate}

It is known \cite[Example 3.2]{CS15} that $L_K(\Gamma)$ is isomoprhic to the Steinberg algebra of the topological groupoid associated with $\Gamma$. The Boolean inverse semigroup of this groupoid (see \cite{JL14} for details)  is the {\em tight Booleanization} ${\mathcal B}_{tight}(I(\Gamma))$ of $I(\Gamma)$. It follows that $L_K(\Gamma)$ is isomorphic to $K\langle {\mathcal B}_{tight}(I(\Gamma)) \rangle$. The image of the canonical map from $I(\Gamma)$ to $L_K(\Gamma)$ is known as the {\em Leavitt inverse semigroup} $LI(\Gamma)$, see \cite{MMW21} for details. 

\subsection{The graph monoid and the isomorphism} The {\em graph monoid} $M_{\Gamma}$ is the abelian monoid given by the generators $\{a_v \colon v \in \Gamma_0\}$, with the defining relations:
\begin{enumerate}
\item[(M)] $a_v=\sum\limits_{e\in \Gamma_1\colon s(e) = v} a_{r(e)}$ for every $v\in \Gamma_0$ that emits edges.
\end{enumerate}  

It was proved by Ara, Moreno and Pardo in \cite[Theorem 3.5]{AMP07} (for a generalization of this theorem to certain classes of algebras of separated graphs, see \cite[Theorem 4.3]{AG12}) that, for a row-finite graph $\Gamma$, there is a natural monoid isomorphism $V(L_K(\Gamma))\simeq M(\Gamma)$. 

\begin{proposition}\label{prop:a} Each ${\mathcal D}$-class of $I(\Gamma)$ has a representative of the form $v$ where $v\in \Gamma_0$. Consequently, ${\mathrm{Typ}}({\mathcal B}_{tight}(I(\Gamma)))$ is $\Gamma_0$-generated. Furthermore, ${\mathrm{Typ}}({\mathcal B}_{tight}(I(\Gamma)))$ is a canonical quotient of $M_{\Gamma}$.
\end{proposition}

\begin{proof} Recall that every element of $I(\Gamma)$ can be written as $x^*y$ where $x$ and $y$ are paths with common ranges (this is well known and follows from (G3)). So every idempotent is ${\mathcal D}$-equivalent to one of the form $(x^*y)^*x^*y = y^*xx^*y = y^*y$.  If the length of $y$ is zero then it is some $v$, and we are done. Other\-wise, we have $r(y) = yy^*  \mathrel{\mathcal D} y^*y$.  
The idempotents of the tight Booleanization ${\mathcal B}_{tight}(I(\Gamma))$ are finite (canonical images of) orthogonal joins of the idempotents of $I(\Gamma)$ and so each ${\mathcal D}$-class has a representative which is a finite orthogonal sum of the vertex idempotents $v$. This means that ${\mathrm{Int}}({\mathcal B}_{tight}(I(\Gamma)))$, and thus also ${\mathrm{Typ}}({\mathcal B}_{tight}(I(\Gamma)))$, is generated by $\Gamma_0$. To show that ${\mathrm{Typ}}({\mathcal B}_{tight}(I(\Gamma)))$ is a canonical quotient of $M_{\Gamma}$ it remains to show that each defining relations (M) of $M_{\Gamma}$ hold in ${\mathrm{Typ}}({\mathcal B}_{tight}(I(\Gamma)))$. It is enough to show that each defining relation (M) holds in ${\mathrm{Int}}({\mathcal B}_{tight}(I(\Gamma)))$. But this is indeed so since, using $r(e)=ee^* \mathrel{\mathcal{D}} e^*e$ and the analogue of (G4) (see \cite{LV21}) for ${\mathcal B}_{tight}(I(\Gamma))$, we have 
$$\bigoplus\limits_{e\in \Gamma_1\colon s(e) = v} [r(e)]  = \bigoplus\limits_{e\in \Gamma_1\colon s(e) = v} [e^*e] = [v].
$$
\end{proof}

We arrive at the following statement which is inspired by \cite[Theorem 7.5]{ABPS21} where a similar result is established for the case where $\Gamma$ is an adaptable separated graph and $I(\Gamma)$ is the inverse semigroup attached to it (for its definition, see \cite{ABPS21}).

\begin{theorem} \label{th:isom_graph}  Let $\Gamma$ be a row-finite graph. Then ${\mathrm{Typ}}({\mathcal B}_{tight}(I(\Gamma)))$ is canonically isomorphic to $M_{\Gamma}$.
\end{theorem}

\begin{proof} We have the following chain of canonical homomorphisms:
$$
M_{\Gamma} \to {\mathrm{Typ}}({\mathcal B}_{tight}(I(\Gamma))) \to V(L_K(\Gamma)) \to M_{\Gamma},
$$
where the first map is due to Proposition \ref{prop:a}, the second one is ${\bf f}$ of Section~\ref{s:connection}, and the third one is the isomorphism of \cite[Theorem 3.5]{AMP07}.
It follows that the first homomorphism $M_{\Gamma} \to {\mathrm{Typ}}({\mathcal B}_{tight}(I(\Gamma)))$ is an isomorphism.
\end{proof}

\section*{Acknowledgement} The author thanks Pere Ara for useful comments.

\end{document}